\documentclass[12pt]{amsart}

\usepackage{amsmath,amssymb,amsbsy,amsfonts,amsthm,latexsym,amsopn,amstext,
                                    amsxtra,euscript,amscd}

\begin{document}




\newfont{\teneufm}{eufm10}
\newfont{\seveneufm}{eufm7}
\newfont{\fiveeufm}{eufm5}
%
%
\newfam\eufmfam
                    \textfont\eufmfam=\teneufm \scriptfont\eufmfam=\seveneufm
                    \scriptscriptfont\eufmfam=\fiveeufm

%
%
\def\frak#1{{\fam\eufmfam\relax#1}}
%


\def\bbbr{{\rm I\!R}} 
\def\bbbc{{\rm I\!C}} 
\def\bbbm{{\rm I\!M}}
\def\bbbn{{\rm I\!N}} 
\def\bbbf{{\rm I\!F}}
\def\bbbh{{\rm I\!H}}
\def\bbbk{{\rm I\!K}}
\def\bbbl{{\rm I\!L}}
\def\bbbp{{\rm I\!P}}
\newcommand{\lcm}{{\rm lcm}}
\def\bbbone{{\mathchoice {\rm 1\mskip-4mu l} {\rm 1\mskip-4mu l}
{\rm 1\mskip-4.5mu l} {\rm 1\mskip-5mu l}}}
\def\bbbc{{\mathchoice {\setbox0=\hbox{$\displaystyle\rm C$}\hbox{\hbox
to0pt{\kern0.4\wd0\vrule height0.9\ht0\hss}\box0}}
{\setbox0=\hbox{$\textstyle\rm C$}\hbox{\hbox
to0pt{\kern0.4\wd0\vrule height0.9\ht0\hss}\box0}}
{\setbox0=\hbox{$\scriptstyle\rm C$}\hbox{\hbox
to0pt{\kern0.4\wd0\vrule height0.9\ht0\hss}\box0}}
{\setbox0=\hbox{$\scriptscriptstyle\rm C$}\hbox{\hbox
to0pt{\kern0.4\wd0\vrule height0.9\ht0\hss}\box0}}}}
\def\bbbq{{\mathchoice {\setbox0=\hbox{$\displaystyle\rm
Q$}\hbox{\raise
0.15\ht0\hbox to0pt{\kern0.4\wd0\vrule height0.8\ht0\hss}\box0}}
{\setbox0=\hbox{$\textstyle\rm Q$}\hbox{\raise
0.15\ht0\hbox to0pt{\kern0.4\wd0\vrule height0.8\ht0\hss}\box0}}
{\setbox0=\hbox{$\scriptstyle\rm Q$}\hbox{\raise
0.15\ht0\hbox to0pt{\kern0.4\wd0\vrule height0.7\ht0\hss}\box0}}
{\setbox0=\hbox{$\scriptscriptstyle\rm Q$}\hbox{\raise
0.15\ht0\hbox to0pt{\kern0.4\wd0\vrule height0.7\ht0\hss}\box0}}}}
\def\bbbt{{\mathchoice {\setbox0=\hbox{$\displaystyle\rm
T$}\hbox{\hbox to0pt{\kern0.3\wd0\vrule height0.9\ht0\hss}\box0}}
{\setbox0=\hbox{$\textstyle\rm T$}\hbox{\hbox
to0pt{\kern0.3\wd0\vrule height0.9\ht0\hss}\box0}}
{\setbox0=\hbox{$\scriptstyle\rm T$}\hbox{\hbox
to0pt{\kern0.3\wd0\vrule height0.9\ht0\hss}\box0}}
{\setbox0=\hbox{$\scriptscriptstyle\rm T$}\hbox{\hbox
to0pt{\kern0.3\wd0\vrule height0.9\ht0\hss}\box0}}}}
\def\bbbs{{\mathchoice
{\setbox0=\hbox{$\displaystyle     \rm S$}\hbox{\raise0.5\ht0\hbox
to0pt{\kern0.35\wd0\vrule height0.45\ht0\hss}\hbox
to0pt{\kern0.55\wd0\vrule height0.5\ht0\hss}\box0}}
{\setbox0=\hbox{$\textstyle        \rm S$}\hbox{\raise0.5\ht0\hbox
to0pt{\kern0.35\wd0\vrule height0.45\ht0\hss}\hbox
to0pt{\kern0.55\wd0\vrule height0.5\ht0\hss}\box0}}
{\setbox0=\hbox{$\scriptstyle      \rm S$}\hbox{\raise0.5\ht0\hbox
to0pt{\kern0.35\wd0\vrule height0.45\ht0\hss}\raise0.05\ht0\hbox
to0pt{\kern0.5\wd0\vrule height0.45\ht0\hss}\box0}}
{\setbox0=\hbox{$\scriptscriptstyle\rm S$}\hbox{\raise0.5\ht0\hbox
to0pt{\kern0.4\wd0\vrule height0.45\ht0\hss}\raise0.05\ht0\hbox
to0pt{\kern0.55\wd0\vrule height0.45\ht0\hss}\box0}}}}
\def\bbbz{{\mathchoice {\hbox{$\sf\textstyle Z\kern-0.4em Z$}}
{\hbox{$\sf\textstyle Z\kern-0.4em Z$}}
{\hbox{$\sf\scriptstyle Z\kern-0.3em Z$}}
{\hbox{$\sf\scriptscriptstyle Z\kern-0.2em Z$}}}}
\def\ts{\thinspace}

\newtheorem{theorem}{Theorem}
\newtheorem{lemma}[theorem]{Lemma}
\newtheorem{claim}[theorem]{Claim}
\newtheorem{cor}[theorem]{Corollary}
\newtheorem{prop}[theorem]{Proposition}
\newtheorem{definition}{Definition}
\newtheorem{question}[theorem]{Open Question}
\newtheorem{remark}[theorem]{Remark}

\def\squareforqed{\hbox{\rlap{$\sqcap$}$\sqcup$}}
\def\qed{\ifmmode\squareforqed\else{\unskip\nobreak\hfil
\penalty50\hskip1em\null\nobreak\hfil\squareforqed
\parfillskip=0pt\finalhyphendemerits=0\endgraf}\fi}

\def\cA{{\mathcal A}}
\def\cB{{\mathcal B}}
\def\cC{{\mathcal C}}
\def\cD{{\mathcal D}}
\def\cE{{\mathcal E}}
\def\cF{{\mathcal F}}
\def\cG{{\mathcal G}}
\def\cH{{\mathcal H}}
\def\cI{{\mathcal I}}
\def\cJ{{\mathcal J}}
\def\cK{{\mathcal K}}
\def\cL{{\mathcal L}}
\def\cM{{\mathcal M}}
\def\cN{{\mathcal N}}
\def\cO{{\mathcal O}}
\def\cP{{\mathcal P}}
\def\cQ{{\mathcal Q}}
\def\cR{{\mathcal R}}
\def\cS{{\mathcal S}}
\def\cT{{\mathcal T}}
\def\cU{{\mathcal U}}
\def\cV{{\mathcal V}}
\def\cW{{\mathcal W}}
\def\cX{{\mathcal X}}
\def\cY{{\mathcal Y}}
\def\cZ{{\mathcal Z}}

\newcommand{\comm}[1]{\marginpar{%
\vskip-\baselineskip 
\raggedright\footnotesize
\itshape\hrule\smallskip#1\par\smallskip\hrule}}





\def\MOV{{\bf{MOV}}}

\hyphenation{re-pub-lished}

\def\dist{\mathrm{dist}}
\def\btau{\overline\tau}
\def\ord{{\mathrm{ord}}}
\def\Nm{{\mathrm{Nm}}}
\def\Sp{{\mathrm{Sp}}}
\renewcommand{\vec}[1]{\mathbf{#1}}

\def \C{{\mathbb{C}}}
\def \F {{\mathbb{F}}}
\def \L{{\mathbb{L}}}
\def \K{{\mathbb{K}}}
\def \Z{{\mathbb{Z}}}
\def \N{{\mathbb{N}}}
\def \Q{{\mathbb{Q}}}
\def\E{{\mathbf E}}
\def\H{{\mathbf H}}
\def\G{{\mathcal G}}
\def\O{{\mathcal O}}
\def\cS{{\mathcal S}}
\def \R{{\mathbb{R}}}
\def\Fp{\F_p}
\def \fp{\Fp^*}
\def\\{\cr}
\def\({\left(}
\def\){\right)}
\def\fl#1{\left\lfloor#1\right\rfloor}
\def\rf#1{\left\lceil#1\right\rceil}

\def\Zm{\Z_m}
\def\Zt{\Z_t}
\def\Zp{\Z_p}
\def\Um{\cU_m}
\def\Ut{\cU_t}
\def\Up{\cU_p}

\def\ep{{\mathbf{e}}_p}
\def\HH{\cH}

\def \Prob{{\mathrm {}}}

\def\LC{{\cL}_{C,\cF}(Q)}
\def\LCn{{\cL}_{C,\cF}(nG)}
\def\Tr{\mathrm{Tr}\,}

\def\taubar{\overline{\tau}}
\def\sigmabar{\overline{\sigma}}
\def\Fn{\F_{q^n}}
\def\En{\E(\Fn)}

\def\mand{\qquad \mbox{and} \qquad}


\title[Distribution of the Number of Points 
on  Curves in  Extensions]{On the Distribution of the Number of Points 
on Algebraic Curves in  Extensions of Finite Fields}

\author{Omran Ahmadi} 
\address{Claude Shannon Institute\\
University College Dublin \\
Dublin 4, Ireland} 
\email{omran.ahmadi@ucd.ie}  

\author{Igor E.~Shparlinski}

\address{Department of Computing\\
Macquarie University\\
Sydney, NSW 2109, Australia} 
\email{igor@ics.mq.edu.au}  
\date{\today}

\maketitle

\begin{abstract} 
Let $\cC$ be a smooth absolutely irreducible curve of genus $g \ge 1$
defined over  $\F_q$, the
finite field of $q$ elements. Let $\# \cC(\F_{q^n})$ be the number
of $\F_{q^n}$-rational points on $\cC$.  Under a certain 
multiplicative independence 
condition on the roots of the zeta-function of $\cC$, 
we derive an asymptotic formula for the
number of $n =1, \ldots, N$ such that 
 $(\# \cC(\F_{q^n}) - q^n -1)/2gq^{n/2}$ belongs to  a given interval $\cI  \subseteq [-1,1]$. This can be 
considered as an analogue of the Sato--Tate distribution 
which covers the case when the curve $\E$ is defined
over $\Q$ and considered modulo consecutive 
primes $p$, although in our scenario the distribution 
function is different.  The above multiplicative independence 
condition has, recently, been considered by E.~Kowalski in 
statistical settings. It 
is trivially satisfied for ordinary elliptic curves and we 
also establish it for a natural family of curves of genus $g=2$. 
\end{abstract}


\section{Introduction}

Let $\cC$  be a smooth absolutely irreducible curve  of genus $g \ge 1$ 
defined over the finite field $\F_q$ of $q$ elements.
We denote by  $\cC(\F_{q^n})$  the set of the $\F_{q^n}$-rational 
points on the projective model of  $\cC$.

By the Weil theorem, 
$$
|  \# \cC(\F_{q^n}) - q^n -1| \le 2g q^{n/2}
$$
see~\cite[Section~VIII.5, Bound~(5.7)]{Lor}, 
however the distribution of values of $ \# \cC(\F_{q^n})$
inside of the  interval $[q^n + 1- 2g q^{n/2}, q^n + 1 + 2g q^{n/2}]$,
in particular, the distribution of the ratios
\begin{equation}
\label{eq:ST ration}
\frac{\#\cC(\F_q) - q -1}{2gq^{1/2}} \in [-1,1].
\end{equation}
is not understood well enough. 

In the case  of elliptic curves $\cC = \cE$ more is known.
For example, the distribution of the  ratios~\eqref{eq:ST ration},
where the 
curve $\cE$ is defined over $\Q$ and reduced modulo
consecutive primes $p$ (that is, $q=p$) is described by the  
{\it Sato--Tate conjecture\/}, which has been recently proven by
R.~Taylor~\cite{Tayl}. In particular, the proportion of primes 
$p \le x$ for which the analogue of the above ratios belongs 
to the interval $[\beta, \gamma]$ is given by 
$$
\mu_{ST}(\beta, \gamma)= \frac{2}{\pi} \int_{\max\{0,\beta\}}^{\min\{1,\gamma\}}
\sqrt{1 - \alpha^2}d\, \alpha, 
$$
as $x\to \infty$.

B.~J.~Birch~\cite{Birch} has also established an analogue of the 
 Sato--Tate conjecture in the dual case when the finite 
field $\F_q$ is fixed and the 
ratios~\eqref{eq:ST ration} are taken 
over all elliptic curves $\E$ over $\F_q$, see also~\cite{Yos}. 

Finally, there is also 
a series of works  showing that a similar type of behavior 
also holds in mixed situations (when both the field and the
curve vary) 
over various families of curves, see~\cite{BaZh,BaSh}.

Here we fix a smooth absolutely irreducible curve 
$\cC$ of genus $g \ge 1$ over $\F_q$
and consider analogues of the  ratios~\eqref{eq:ST ration} taken  
in the consecutive extensions of $\F_q$, 
that is, for $\F_{q^n}$-rational points on $\cC$.
We write the cardinality $\#\cC\(\F_{q^n}\)$ of 
the sets of $\F_{q^n}$-rational points 
on the projective model of  $\cC$ as
$$
\# \cC(\F_{q^n}) = q^n + 1 - a_n.
$$
and study the distribution of the ratios
\begin{equation}
\label{eq:alpha}
\alpha_n = \frac{a_n}{2gq^{n/2}} \in [-1,1], 
\qquad n =1,2, \ldots\,. 
\end{equation}
Under a certain additional condition of multiplicative 
independence of so-called {\it Frobenius eigenvalues\/}
of $\cC$ we obtain an asymptotic formula for the distribution
function of the ratios~\eqref{eq:alpha}. By a result
of E.~Kowalski~\cite{Kow}, the additional condition needed for our proofs to work is satisfied for 
a wide class of curves. We are also grateful to Nick Katz 
for the observation that a result of N.~Chavdarov~\cite{Chav}
can also be used to show that the desired property 
holds for a ``typical'' curve. 

 Here, we also show that this additional condition holds
for so called {\it ordinary\/} 
elliptic curves and ordinary smooth curves of genus $g=2$ whose Jacobians are absolutely  simple. The latter result can be of independent interest.

In particular, for $g=1$, that is, 
when $\cC= \cE$ is an ordinary elliptic curve, 
our result implies that the distribution of the ratios~\eqref{eq:alpha} 
is not governed by $\mu_{ST}(\beta, \gamma)$ but rather 
by a different distribution function
\begin{equation}
\label{eq:gamma}
\lambda_1(\beta,\gamma) = \frac{1}{\pi} \int_{\max\{0,\beta\}}^{\min\{1,\gamma\}}
(\sqrt{1-{\alpha^2}})^{-1}d\, \alpha.
\end{equation}
We also remark that for supersingular elliptic curves 
we have $a_n =0$ for every odd $n$ 
(and $a_n = 2q^{n/2}$ for every even $n$).

Throughout this paper, the implied constants
in  the symbols `$O$' and `$\ll$'
may depend on the base field $\F_q$
(we recall that
$A\ll B$ and $B\gg A$ are equivalent to $A=O(B)$).

\section{Main Results} 

\subsection{Frobenius Angles}

We refer to~\cite{Lor} for a background on curves and their zeta-functions.

For    a smooth projective curve $\cC$  over the finite field $\F_q$ we define the 
zeta-function of $C$ as 
$$
Z(T)=\exp\(\sum_{n=1}^{\infty}\#\cC(\F_{q^n})\frac{T^n}{n}\).
$$
It is well-known that if $\cC$ is of genus $g \ge 1$ then 
$$
Z(T)=\frac{P(T)}{(1-t)(1-qt)},
$$
where $P(T)$ is a polynomial of degree $2g$ with integer coefficients.

Furthermore we have
$$
P(t)=\prod_{j=1}^{2g}(1-\tau_j T), 
$$
where 
$\tau_1, \ldots, \tau_{2g}$ are called the {\it 
Frobenius  eigenvalues\/} and satisfy 
\begin{equation}
\label{eq:Frob}
|\tau_j|  = q^{1/2}, \quad \tau_{j+g} = \btau_j, \qquad 
j =1, \ldots, g,
\end{equation}
(where $\btau$ means complex conjugate of $\tau$), 
see~\cite[Section~VIII.5]{Lor}.

If $C$ is considered over the degree $m$ extension of $\F_q$, then
\begin{equation}
\label{eq:rat fun}
Z_m(T)=\exp\(\sum_{n=1}^{\infty}\#\cC(\F_{q^{mn}})\frac{T^n}{n}\) 
= \frac{P_m(T)}{(1-T)(1-qT)},
\end{equation}
where
$$
P_m(T)=\prod_{j=1}^{2g}(1-\tau_j^m T). 
$$
It is also well-known that
\begin{equation}
\label{eq: Card}
  \# \cC(\F_{q^n}) = q^n + 1 - \sum _{j=1}^{2g}\tau_j^n, 
\end{equation}
which underlies our method. 

Furthermore, we write~\eqref{eq:Frob} as
\begin{equation}
\label{eq:theta} 
\tau_j =  q^{1/2} e^{\pi i \vartheta_j} \mand
\tau_{j+g}=  q^{1/2} e^{-\pi i \vartheta_j},
\end{equation} 
with some $\vartheta_j \in [0,1]$, $j =1, \ldots, g$, which we call
{\it Frobenius angles\/}. (sometimes $\pi\vartheta_j$ is called a Frobenius angle)  

We show that if $1, \vartheta_1, \ldots, \vartheta_g$ are linearly 
independent over $\Z$, or alternatively $q^{1/2}, \tau_1, \ldots, \tau_g$
are multiplicatively independent over $\Z$, then the ratios~\eqref{eq:alpha}
are distributed in accordance with the 
distribution function $\lambda_g(\beta,\gamma) $
which for $g=1$ is defined by~\eqref{eq:gamma} and then recursively
as 
\begin{equation}
\label{eq:gamma g}
\lambda_g(\beta,\gamma) = 
\frac{1}{\pi} \int_0^1 
\lambda_{g-1}\(g\beta-\cos( \pi \alpha),  g\gamma-\cos( \pi \alpha)\)d\, \alpha
 \end{equation}
and give an estimate on the error term. 

\subsection{Distribution of the Number of Points in Extensions}

For a fixed absolutely irreducible curve 
$\cC$ over $\F_q$,
let $T_{\beta,\gamma}(N)$ be the number 
of ratios $\alpha_n \in [\beta,\gamma]$ for $n=1, \ldots, N$.

We say that real numbers 
$\psi_1, \ldots, \psi_s$  are linearly independent
modulo $1$ if and only if $1, \psi_1, \ldots, \psi_s$  are linearly  independent over $\Z$.

\begin{theorem}
\label{thm:IndivBound}
Suppose $\cC$ is a smooth absolutely irreducible curve  of genus $g \ge 1$ 
over $\F_q$ such that $\vartheta_1, \ldots, \vartheta_g$ are linearly 
independent modulo $1$. Then there is a constant $\eta>0$ depending only on $q$ and $g$ such that uniformly over $-1 \le \beta \le \gamma \le 1$ we have
$$
T_{\beta,\gamma}(N) = \lambda_g(\beta,\gamma) N + 
O(N^{1-\eta}),
$$
where $\lambda_g$ is defined by~\eqref{eq:gamma g}. 
\end{theorem}

\subsection{Linear Independence of Frobenius Angles}

We say that  $\cC$ is an ordinary curve if and only if 
at least half of the Frobenius  eigenvalues
$\tau_1, \btau_1, \ldots, \tau_g, \btau_g$ 
are $p$-adic units where $p$ is the characteristic 
of $\F_q$, see~\cite[Definition~3.1]{Howe} for several 
equivalent definitions. Notice that all the Frobenius eigenvalues are 
algebraic numbers. So whenever they are considered as $p$-adic numbers it is
implied that we have chosen an embedding of $\overline{\mathbb{Q}}$ in 
${\overline{\mathbb{Q}}}_p$ allowing us to view Frobenius eigenvalues as 
$p$-adic numbers.

\begin{theorem}
\label{thm:LinIndep}
Suppose $\cC$ is a smooth projective curve of genus $g \le 2$ 
over $\F_q$. Furthermore suppose that $\cC$ is either an ordinary elliptic curve
or is a curve of genus $g=2$ which has an absolutely simple Jacobian. 
Then the Frobenius angles $\vartheta_1, \ldots, \vartheta_g$ are linearly 
independent modulo $1$. 
\end{theorem}

\section{Proof of Theorem~\ref{thm:IndivBound}}

\subsection{Preparation}
 
 We see from~\eqref{eq: Card} and~\eqref{eq:theta} that
\begin{equation}
\label{eq:Frob_roots}
\alpha_n = \frac{1}{g} \sum_{j=1}^{g}\cos( \pi \vartheta_j n).
\end{equation}

Therefore, denoting via $\cV_g(\beta, \gamma)$
the $g$-dimensional domain consisting of 
points $(\psi_1, \ldots, \psi_g) \in [0,1]^g$ 
such that 
$$
\beta \le \frac{1}{g} \sum_{j=1}^{g}\cos( \pi \psi_j) \le \gamma
$$
we see that 
\begin{equation}
\label{eq:T V}
T_{\beta,\gamma}(N)  = 
\# \{n =1, \ldots, N~:~ 
 (\vartheta_1n , \ldots, \vartheta_gn ) \in \cV_g(\beta, \gamma)\}. 
\end{equation}
Thus it is natural that we use tools from the theory of uniformly distributed 
sequences to estimate $T_{\beta,\gamma}(N)$.

\subsection{Background on the Uniform Distribution}

For a finite set $\cF \subseteq [0,1]^s$ of the $s$-dimensional
unit cube, we define its {\it discrepancy
with respect to a domain $\varXi \subseteq [0,1]^s$\/}
as
$$
\varDelta(\cF ,\varXi) = \left| \frac{ \#\{ \vec{f}\in\cF :\ \vec{f}\in
\varXi\} }{\#\cF} - \mu(\varXi)\right|,
$$
where $\mu$ is the Lebesgue measure on $ [0,1]^s$.

We now define  the {\it  discrepancy\/} of $\cF $ as
$$
D(\cF) = \sup_{\varPi  \subseteq [0,1]^s}  \varDelta(\cF ,\varPi) ,
$$
where the supremum is taken over all boxes $\varPi = [\alpha_1,
\beta_1) \times \ldots \times [\alpha_s, \beta_s) \subseteq [0,1]^s$,
see~\cite{DrTi,KuNi}.

We define the  distance between a vector $\vec{u} \in [0,1]^s$
and a set $\varXi\subseteq [0,1]^s $  by
$$
\dist(\vec{u},\varXi) = \inf_{\vec{w} \in\varXi}
\|\vec{u} - \vec{w}\|,
$$
where  $\|\vec{v}\|$ is the Euclidean norm of $\vec{v}$. Given
$\varepsilon >0$ and a domain  $\varXi \subseteq [0,1]^s $ we define
the  sets
$$
\varXi_\varepsilon^{+} = \left\{ \vec{u} \in  [0,1]^s \backslash
\varXi \ : \ \dist(\vec{u},\varXi) < \varepsilon \right\}
$$
and
$$
\varXi_\varepsilon^{-} = \left\{ \vec{u} \in \varXi \ : \
\dist(\vec{u},[0,1]^s \backslash \varXi )  < \varepsilon  \right\} .
$$

Let $h(\varepsilon)$ be an arbitrary increasing function defined for
$\varepsilon >0$ such that
$$
\lim_{ \varepsilon \to 0}h(\varepsilon) = 0.
$$
As in~\cite{Lac,NiWi}, we define the
class $\cS_h$ of  domains  to include domains  $\varXi \subseteq [0,1]^s $ for which
$$
\mu\(\varXi_\varepsilon^{+} \)\le h(\varepsilon)
\qquad \mbox{and}
\qquad
\mu\(\varXi_\varepsilon^{-} \)\le h(\varepsilon) .
$$

A relation  between $D(\cF)$ and $\varDelta(\cF ,\varXi)$
for $\varXi \in \cS_h$ is given by the following inequality of~\cite{Lac}
(see also~\cite{NiWi}).

\begin{lemma}
   \label{lem:LNW bound} For any domain  $\varXi \in \cS_h$, we have
$$
\varDelta(\cF ,\varXi) \ll h\(s^{1/2} D(\cF )^{1/s}\)  .
$$
\end{lemma}

Finally,  the following bound, which is a special case
of a more general result of H.~Weyl~\cite{Weyl}
shows that if  $\varXi$ has a piecewise smooth boundary
then $\varXi \in \cS_h$ for some   function
$h(\varepsilon) = O\(\varepsilon\)$.

\begin{lemma}
   \label{lem:Weyl bound} For any domain   $\varXi \in [0,1]^s$ with
piecewise smooth boundary, we
have
$$
\mu\(\varXi_\varepsilon^{\pm }\) = O(\varepsilon).
$$
\end{lemma}

Clearly the domain $\cV_g(\beta, \gamma)$ satisfies the condition of 
Lemma~\ref{lem:Weyl bound}. Thus we see from Lemma~\ref{lem:Weyl bound}
that we need to estimate the discrepancy of the points
\begin{equation}
\label{eq:Points}
 (\vartheta_1 n , \ldots, \vartheta_gn ), \qquad
 n =1, \ldots, N.
\end{equation}

This is a well-studied question, however
the answer depends on the Diophantine 
properties of $\vartheta_1,  \ldots, \vartheta_g$, 
see~\cite{DrTi,KuNi}. More specifically, 
we denote by $\R^+$ the set of positive real numbers 
and by $\|z\|$ the distance between a real $z$ and the 
closest integer. We now recall~\cite[Theorem~1.80]{DrTi}

\begin{lemma}
   \label{lem:Discr} Suppose that $\psi_1, \ldots, \psi_s$ 
   are linearly 
independent modulo $1$ and for some continuous 
function $\varphi(t): \R^+ \to \R^+$ such that 
$\varphi(t)/t$ is monotonically increasing for real $t\ge 1$ we
have
$$
\|k_1 \psi_1 + \ldots + k_s \psi_s\| \ge 
\frac{1}{\varphi\(\max\{|k_1|, \ldots, |k_s|\}\)}
$$
for any non-zero  vector $(k_1, \ldots, k_s) \in \Z^s$. 
Then the discrepancy $D(N)$ of the sequence
$$
(\psi_1 n , \ldots, \psi_sn ), \qquad
n =1, \ldots, N, 
$$
satisfies
$$
D(N) \ll \frac{\log N \log \varphi^{-1}(N)}{\varphi^{-1}(N)},
$$
where $ \varphi^{-1}(t)$ is the inverse function of $ \varphi(t)$. 
 \end{lemma}

\subsection{Linear Forms in Logarithms and Linear Combinations of Frobenius Angles}

We  present a classical result of A.~Baker~\cite{Bak}
in  the following greatly simplified form (see also~\cite{BakWu,Mat} 
for more recent achievements, which can be used to make our 
estimates more explicit).

\begin{lemma}
 \label{lem:LinFormLog} For arbitrary 
algebraic numbers $\xi_1 \ldots, \xi_s$ 
there are   constants $C_1>0$ and $C_2> 1$ such that 
the inequality 
\begin{equation*}
\begin{split}
0<  |\xi_1^{k_1}  & \ldots   \xi_s^{k_s} -1 | \\
& \le C_1 \(\max\{|k_1|, \ldots, |k_s|\}+1\)^{-C_2}
\end{split}
\end{equation*}
has no solution in  $(k_1, \ldots, k_s) \in \Z^s\setminus(0,\ldots,0)$. 
 \end{lemma}

 We are now ready to establish a necessary result 
 which is needed for an application of Lemma~\ref{lem:Discr}.
 
\begin{lemma}
 \label{lem:Type} There are   constants $c_1>0$ and $c_2>1$ 
depending only on $q$ and $g$ such that 
if  $\cC$ is a smooth projective curve  of genus $g \ge 1$ 
over $\F_q$ and
$\vartheta_1, \ldots, \vartheta_g$ are linearly 
independent modulo $1$ then
$$
\|k_1 \vartheta_1 + \ldots + k_g \vartheta_g\| \ge c_1
\(\max\{|k_1|, \ldots, |k_g|\}+1\)^{-c_2}
$$
for any  non-zero  vector $(k_1, \ldots, k_g) \in \Z^g$. 
\end{lemma}

\begin{proof} Suppose that for some integer $k_0$ we 
have 
$$
k_1 \vartheta_1 + \ldots + k_g \vartheta_g - k_0 = \delta
$$
where $\delta \in [-1/2, 1/2]$.
Then, recalling~\eqref{eq:theta}, we derive
$$
\tau_1^{2k_1}\ldots \tau_g^{2k_g} =q^{k_1+\ldots+ k_g} e^{- \pi i \delta}
$$
We see that because of the linear independence of 
$1, \vartheta_1, \ldots, \vartheta_g$ we have
$\delta>0$. We can assume that $\delta$ is sufficiently small
(as otherwise there is nothing to prove). 
Hence 
$$
0 < \left|\tau_1^{2k_1}\ldots \tau_g^{2k_g}q^{-k_1-\ldots- k_g} -1 \right| = \left|e^{- \pi i \delta} -1 \right| =
\pi \delta + O(\delta^2) \le 4 \delta
$$
Applying Lemma~\ref{lem:LinFormLog}, we obtain the desired 
result with the constants  $c_1$ and $c_2$ depending on $\vartheta_1, \ldots, \vartheta_g$.
However, examining the zeta-function of the curve $\cC$, see~\cite[Section~VIII.5]{Lor}, 
we conclude for each $q$ and $g$ there are only finitely many possibilities 
for $\vartheta_1, \ldots, \vartheta_g$ thus $c_1$ and $c_2$ can 
be made to depend only on $q$ and $g$. 
\end{proof}

\subsection{Concluding the Proof}

We see from Lemma~\ref{lem:Type} that Lemma~\ref{lem:Discr} 
applies to the discrepancy $\Delta(N)$ of the points~\eqref{eq:Points}
with $\varphi(t) = c_1 (t+1)^{c_2}$, thus
\begin{equation}
\label{eq:Discr}
\Delta(N) \le N^{-\kappa}, 
\end{equation}
where $\kappa$ depends only on $\cC$. 
As we have mentioned, the domain $\cV_g(\beta, \gamma)$ 
satisfies the condition of 
Lemma~\ref{lem:Weyl bound}. 
Recalling~\eqref{eq:T V} and  combining the bound~\eqref{eq:Discr} 
with Lemmas~\ref{lem:LNW bound} and~\ref{lem:Weyl bound}, 
we obtain 
$$
T_{\beta,\gamma}(N) =  N \mu\(\cV_g(\beta, \gamma)\) + 
O\(N^{1-\eta}\),
$$
where $\mu$ denotes the Lebesgue measure on $ [0,1]^g$.
Since 
$$
\mu\(\cV_g(\beta, \gamma)\)  = \lambda_g(\beta, \gamma)
$$
we conclude the proof.

\section{Proof of Theorem~\ref{thm:LinIndep}}

\subsection{Elliptic Curves}

In the case of $g=1$ the result follows immediately
from~\cite[Lemma~2.5]{LuSh}. However here we present 
a more general statement which maybe  of independent 
interest.

 \begin{lemma}
\label{lem:MultIndep-general}
Suppose that $\cC$ is an ordinary smooth projective curve  of genus $g \ge 1$ 
over $\F_q$. Then all Frobenius angles  
$\vartheta_1, \ldots, \vartheta_g$ are irrational. 
\end{lemma}

\begin{proof} Since for every pair of conjugated Frobenius eigenvalues we
have $\tau_j \btau_j = q$, then for $j=1, \ldots, g$
from $\tau_j$ and $\btau_j$ at most one can be a $p$-adic unit. From the  definition of an ordinary 
curve we conclude that for every $j=1, \ldots, g$, exactly one 
from  $\tau_j$ and $\btau_j$, is a $p$-adic unit. 

However,  if $\vartheta_i = r/s$ is a rational Frobenius angle, 
then for the corresponding Frobenius eigenvalues we obtain 
$$
\tau_i^{2s} = \btau_i^{2s} = q^{s}.
$$
Thus neither of $\tau_i$ and $\btau_i$ can be a $p$-adic unit. This  contradiction concludes the proof. 
\end{proof}

\subsection{Curves of Genus $g=2$}

Every curve of genus $2$ is either ordinary or supersingular or of $p$-rank 1.
(equivalently the Jacobian of the curve is either ordinary or supersingular or
of $K3$ type);  for example, this 
follows from~\cite[n.~133, pp.~1--20]{Serre}. 
If a 
curve of genus $g=2$ is supersigular, then its Jacobian over 
$\overline{\F}_q$ is isogenous to a product of supersingular elliptic curves and
cannot be absolutely simple. On the other hand, if a curve is of $p$-rank 1 and its
Jacobian is absolutely simple, then the claim follows from a result of 
Yu.~G.~Zarhin~\cite{Zarhin}. Thus the only remaining case is the case of ordinary curves
which we deal with in the following. 

In fact here we   establish the linear independence of $\vartheta_1,\vartheta_2$ modulo $1$ for 
any   two Frobenius angles $\vartheta_1,\vartheta_2$ (with 
$\vartheta_1 \ne  \pm \vartheta_2$) for 
an ordinary smooth projective curve of arbitrary  genus $g\ge 2$ which has an
absolutely simple Jacobian. It is slightly more general than what we need for the
proof of Theorem~\ref{thm:LinIndep} and can be of independent interest.

First of all, we need the following property of the numerator $P_m(T)$ 
in~\eqref{eq:rat fun}
of the zeta-function of $\cC$ which is a result of Honda-Tate theory for ordinary
varieties (see~\cite[Theorem~3.3]{Howe}).

\begin{lemma}\label{lem:irred}
If an ordinary smooth projective curve of arbitrary  genus $g\ge 2$ has an
absolutely simple Jacobian, then the polynomials $P_m(T)$ are irreducible 
over $\Z$ for every $m =1,2, \ldots$.
\end{lemma}

We are  now ready to prove our principal result.

\begin{lemma}\label{lem:g-independence}
Let $\cC(\F_q)$ be an ordinary smooth projective curve of genus $g \ge 2$
which has an absolutely simple Jacobian.
Then $\vartheta_1, \vartheta_2$  are linearly independent modulo $1$
for any two Frobenius angles $\vartheta_1$ and $\vartheta_2$ .
\end{lemma}

\begin{proof}
Suppose $\vartheta_1$, $\vartheta_2$ and $1$ are $\Z$-linear dependent and for some integers $u$, $v$ and $w$ which are not zero simultaneously we have $u\vartheta_1+v\vartheta_2+w=0$. We know that for an ordinary curve two of the eigenvalues corresponding to two different 
angles are $p$-adic units in  $\overline{\Q}_p$
 (see the proof of Lemma~\ref{lem:MultIndep-general}). If we assume that
$\tau_1=\sqrt{q}e^{\pi i\vartheta_1}$ and $\tau_2=\sqrt{q}e^{\pi i\vartheta_2}$ are $p$-adic units in $\overline{\Q}_p$, then from $u\vartheta_1+v\vartheta_2+w=0$ it follows that
$$
\tau_1^{2u}\tau_2^{2v}e^{i\pi 2w}=q^{u+v}.
$$
Now since $\tau_1$ and $\tau_2$ are $p$-adic units in $\overline{\Q}_p$ we have $u+v=0$. This
along with $u\vartheta_1+v\vartheta_2+w=0$ implies that either $u=v=0$ or $\vartheta_1-\vartheta_2=w/v$. If the former case happens, then we have $w=0$ which is a contradiction. If the latter happens, then we have $\tau_1^{2v}=\tau_2^{2v}$. This means that 
the numerator of the zeta-function of $\cC(\F_{q^{2v}})$ has a double root 
and hence it splits, which is a contradiction since we have assumed that the 
Jacobian of the curve is absolutely simple and so the numerator of 
its zeta-function remains irreducible when considered over any extension of $\F_q$.

On the other hand if we assume that  $\tau_1$ and $\btau_2$ are $p$-adic units, then from $u\vartheta_1+v\vartheta_2+w=0$ it follows that
$$
\tau_1^{2u}\btau_2^{-2v}e^{i\pi 2w}=q^{u-v}.
$$
Now since $r_1$ and $\btau_2$ are $p$-adic units in $\overline{\Q}_p$, we have $u=v$. This
along with $u\vartheta_1+v\vartheta_2+w=0$ implies that either $u=v=0$ and hence $w=0$ or $\vartheta_1+\vartheta_2=-w/v$. The former is a contradiction and 
the latter is impossible since it would imply that $\tau_1^{2v}\tau_2^{2v}=1$ which in turn implies that $\tau_2$ is a $p$-adic unit too while this cannot happen as from $\tau_2$ and $\btau_2$ exactly one of them is a $p$-adic unit. The remaining cases can be dealt with similarly.
\end{proof} 

\section{Comments}

\subsection{Statistics of linear independence of Frobenius angles
in families of curves}
\label{sec:stat-indep}
When $\cC$ is a smooth projective ordinary curve of genus  $g \ge 3$ 
over $\F_q$,
then it seems to be more subtle to establish the linear 
independence modulo $1$ of the Frobenius angles of $\cC$. 
E.~Kowalski~\cite{Kow} 
gives a statistical result  
that Frobenius angles of most algebraic curves from 
a certain natural family are linear independent modulo $1$.  One can also find in~\cite{Kow} 
examples of Abelian varieties or curves with  $\Z$-linear dependent  Frobenius angles. 
However, the examples given in~\cite{Kow} correspond to Abelian varieties which
are neither ordinary nor absolutely simple. We think it is natural to conjecture that
Frobenius angles of a smooth projective ordinary curve $\cC$ whose Jacobian is absolutely simple are linearly independent modulo $1$. Notice that in~\cite{HowZhu} it has been shown that most of the smooth 
curve are ordinary and have an absolutely simple Jacobian. 
A natural way to attack this conjecture would be to investigate 
the Galois group of the numerator $P(T)$ of the zeta-function 
of such curves and employ methods of~\cite{Kow}.  
Finally it is worth mentioning the following independence result 
from~\cite{KrSc} attributed to B.~Poonen which actually goes back to an earlier
paper by M. Spie\ss~\cite{Spi}:
If $\cE_1, \ldots, \cE_k$ are pairwise absolutely non-isogenous elliptic curves over the 
finite field $\F_q$ and $\tau_i$ is a Frobenius eigenvalue of $E_i$, then 
$\tau_1,\ldots,\tau_k$ are multiplicatively independent and hence the corresponding Frobenius angles are $\Z$-linearly independent. 
This result may be used to study joint distributions
of  $\#\cE_1(\F_q^n), \ldots, \#\cE_k(\F_q^n)$.
Notice that the above result implies that if the Jacobian of a curve of genus $g$ is isogenous to the direct product of $g$ pairwise absolutely non-isogenous elliptic curves,
then the Frobenius angles of the curve are linearly independent
modulo $1$. Another
independence result can be found in~\cite{Len-Zar}.

Another ``statistical'' approach to linear 
independence modulo $1$ of the Frobenius angles
stems from the work of N.~Chavdarov~\cite{Chav}, which 
asserts that   for a fixed genus $g \ge 1$, as $q$ grows, 
the numerator $P(T)$ of the zeta-function $Z(T)$ of ``most'' 
curves of genus $g$ over $\F_q$
is irreducible over $Q$, and, furthermore, the Galois group
of $P(T)$ is the {\it Weyl group\/} $\cW_g$ of the symplectic group  
$\Sp(2g)$.   Now, suppose $\cC$ be such an ordinary smooth absolutely irreducible 
curve of genus $g \ge 1$
defined over  $\F_q$,  (that is,  the Galois group
of $P(T)$ is $\cW_g$).  
Assume the Frobenius angles of  $\cC$ whose Jacobian is absolutely simple
are linearly dependent modulo $1$, that is
\begin{equation}
\label{eq:Rel1}
\prod_{i=1}^g \tau_i^{k_i} = q^{k_0}
\end{equation}
for some nonzero vector $(k_0, \ldots, k_g) \in \Z^{g+1}$.
Since the the Weyl group contains a subgroup generated 
by transpositions of $(\tau_i, \overline{\tau_i})$,  $i=1, \ldots, g$,
see the discussion in~\cite[Section~5]{Chav},
we infer  that for each $i=1, \ldots, g$, 
 we have  an automorphism $\eta_i \in \cW_g$, 
which is complex
conjugation on $\tau_i$ but which is the identity on any $\tau_j$ 
and $\overline{\tau_j}$ with $j \ne i$. 

Clearly if $k_1=\ldots = k_g=0$ then also $k_0=0$. 
Thus without loss of generality we can assume that 
$k_1 \ne 0$. 
Applying the automorphism $\eta_1$, we obtain a new relation
\begin{equation}
\label{eq:Rel2}
\overline{\tau_1}^{k_1}\prod_{i= 2}^g\tau_i^{k_i} = q^{k_0}.
\end{equation}
From~\eqref{eq:Rel1} and ~\eqref{eq:Rel2} we infer that 
$$
\overline{\tau_1}^{k_1} = \tau_1^{k_1}.
$$
Multiplying both sides by $\tau_1^{k_1}$, we deduce
$$
\tau_1^{2k_1}=q^{k_1}.
$$
Thus $\tau_1 = q^{1/2}\rho$ for  a root of unity $\rho$, 
contradicting ordinarity, (see, for example, 
Lemma~\ref{lem:MultIndep-general}). 

\subsection{Possible generalisations}

Our method also applies to studying the distribution of the number
of points on the Jacobians $J_\cC(\F_{q^n})$ of a given curve $\cC$ 
(certainly for an elliptic curve
 $\cC = \cE$ it is the same question as the question of studying $\cE(\F_{q^n})$ 
 and this is
covered by Theorems~\ref{thm:IndivBound} and~\ref{thm:LinIndep}). 
For example, we recall that 
$$
\# J_\cC(\F_{q^n})=\prod_{j=1}^{2g} \(1 - \tau_j^n\) 
$$
see~\cite[Corollary~VIII.6.3]{Lor}.

We note that the analogue  of the Sato-Tate conjecture 
for elliptic curves is also believed to be true
for Kloosterman sums, 
see~\cite{Adolph,ChaLi,FoMic1,FoMic2,FMRS,Katz,KatzSar,Lau,Mich1,Mich2,Nied}
for various 
modifications and generalizations of this conjecture and further references. 
However, in this case the original conjecture is
still open as the result of R.~Taylor~\cite{Tayl}
does not seem to apply to Kloosterman sums. 

For $a \in \F_q^*$ we consider the Kloosterman 
sum
$$
K_{q^n}(a) = \sum_{x\in \F_{q^n}^*} \psi\(\Tr_{\F_{q^n}/\F_q}\(x + ax^{-1}\)\), 
$$
where $\psi$ is a fixed nonprincipal additive character of $\F_q$ 
and 
$$
\Tr_{\F_{q^n}/\F_q}(z) = \sum_{j=0}^{n-1} z^{q^j}
$$
is the trace of $z \in \F_{q^n}$ in $\F_q$.  
We have
$$
\left|K_{q^n}(a) \right| \le 2 q^{n/2}, \qquad a \in \F_q^*. 
$$
see~\cite[Theorem~11.11]{IwKow}. 
Therefore,  again, for a fixed $a \in \F_q^*$ we can  define  and study the sequence
$$
\kappa_n =  \frac{ K_{q^n}(a)}{ 2 q^{n/2}} \in [-1,1],
\qquad n =1,2, \ldots\,. 
$$
Since we have the analogue of~\eqref{eq: Card}, that is, 
$$
K_{q^n}(a) =  \sigma^n + \sigmabar^n,
$$
for some complex
conjugate quadratic
irrationalities $\sigma$ and $\sigmabar$ with
$$
|\sigma| = |\sigmabar| = q^{1/2},
$$
see~\cite[Section~11.7]{IwKow} our arguments apply to Kloosterman sums as well.

\section*{Acknowledgements}

We are grateful to Nick Katz for explaining to us the argument
linking linear independence modulo $1$ of the Frobenius angles  
with the work of  N.~Chavdarov~\cite{Chav}, 
which is presented in Section~\ref{sec:stat-indep}.

We would also like to thank Emmanuel Kowalski and Alexey Zaytsev for very insightful discussions and Sandro Mattarei for a very careful reading of the manuscript. 

The comments and suggestions of the anonymous referee have been 
very helpful as well.

This paper was initiated during a very enjoyable visit of I.~S. at
the Department of Combinatorics \& Optimization of the University
of Waterloo whose hospitality, support and stimulating
research atmosphere are gratefully appreciated.
Research of O.~A is supported by the 
Claude Shannon Institute, Science Foundation Ireland Grant 06/MI/006
 and research of I.~S. was supported by ARC grant DP0556431.

  \end{document}